\documentclass[11pt,twoside,reqno]{amsart}
\usepackage{amsmath}
\usepackage{amsthm}
\usepackage{amsfonts,amssymb}
\usepackage{bbm}
\usepackage{latexsym}
\usepackage{mathrsfs}
\usepackage[all]{xy}
\usepackage{url}
\usepackage[pdftex]{graphicx}
\usepackage{float}
\usepackage{hyperref}
\usepackage{amssymb,yfonts}
\usepackage{fontenc}

\usepackage{mathdots}

\usepackage[boxsize=23pt]{ytableau}

\usepackage[utf8]{inputenc}
\usepackage[english]{babel}
\usepackage{lipsum}
\usepackage{amsthm}
\usepackage{todonotes}
\usepackage[vcentermath]{youngtab}

\theoremstyle{plain}
\newtheorem{lema}{Lemma}[section]
\newtheorem{prop}[lema]{Proposition}
\newtheorem{teo}[lema]{Theorem}
\newtheorem{conj}[lema]{Conjecture}
\newtheorem{ques}[lema]{Question}

\theoremstyle{remark}

\newtheorem{remark}[lema]{Remark}
\newtheorem{obs}[lema]{Observation}

\theoremstyle{definition}
\newtheorem{defi}[lema]{Definition}
\newtheorem{ej}[lema]{Example}

\def\a\Si{{\rm{a}\Sigma }}
\def\w\Si{{\rm{w}\Sigma }}

\def\w {{\textrm {w}}}

\def\max{{\text{max}}}
\def\des{{\mathrm{des}}}

\def\a{\mathrm{a}}
\def\Int{{\mathrm{Int}}}

\def\Dec{\mathrm{Dec}}

\def\D{\Delta}

\def\BB{\mathbb{B}^+_{n,j}}
\def\B{{B}^+_{n,j}}

\def\Bj{{B}^+_{n,n-j+1}}
\def\BBB{\widetilde{\mathbb{B}}^+_{n,j}}

\def \g{\gamma}
\def \ga{\widetilde{\gamma}}

 \newcommand{\Si}{\Sigma}
 \newcommand{\si}{\sigma}

\begin{document}

\title[$\g$- and Local $\g$-vectors ]{ On $\g$- and local $\g$-Vectors of the Interval Subdivision }

\author[I. Anwar]{Imran Anwar}
\author[S. Nazir]{Shaheen Nazir}

\address{Abdus Salam School of Mathematical Sciences\\
Government College  University\\
Lahore, Pakistan}

\address{Department of Mathematics\\
Lahore University of Management Sciences\\
Lahore, Pakistan}

\email{imrananwar@sms.edu.pk}

\email{shaheen.nazir@lums.edu.pk}

\begin{abstract}
We show that the $\g$-vector of the interval subdivision of a simplicial complex
with a nonnegative and symmetric $h$-vector is nonnegative. In particular, we prove
that such $\g$-vector is the $f$-vector of some balanced simplicial complex.
Moreover, we show that the local $\g$-vector of the interval subdivision of a simplex
is nonnegative; answering a question by Juhnke-Kubitzke et al.%, \cite{Juhnke-Kubitzke2018}.

\end{abstract}\subjclass[2010]{}


\thanks{The authors are grateful to Volkmar Welker  for suggesting the study of 
nonnegativity of $\g$-vector of the interval  subdivision.
This project is
supported by the Higher Education Commission of Pakistan. }

\maketitle
%\tableofcontents

\section{Introduction}
The $\g$-vector is an important enumerative invariant
 of a flag homology sphere. The general question is asked about $\g$-vector
 whether it is nonnegative or not. It has been conjectured by Gal in \cite{gal2005real} that this
vector is nonnegative for every such sphere.
\begin{conj}\label{Gal} \cite{gal2005real}
  If $\D$ is a flag homology sphere, then $\g(\D)$ is nonnegative.
\end{conj}
Conjecture \ref{Gal} is a strengthening of the well known
Charney-Devis conjecture. The Gal conjecture holds for all Coxeter
complexes (see \cite{stembridge2008coxeter}), for the dual
simplicial complexes of associahedron and cyclohedron (see
\cite{nevo2011}), and for barycentric subdivision of homology
sphere (see\cite{nevo2011gamma}). The authors in \cite{nevo2011}
conjectured  further strengthening of Gal conjecture.
\begin{conj}\label{NP}\cite[Problem 6.4]{nevo2011}
  If $\D$ is a flag homology sphere then $\g(\D)$ is the $f$-vector of some balanced simplicial complex.
\end{conj}
\noindent This conjecture holds for the dual simplicial complex
of all flag nestohedera, see in  \cite{aisbett2014frankl}. Frohmader \cite[Theorem 1.1]{frohmader2008face}
showed that the $f$-vector of any flag simplicial complex satisfies the Frankl-F\"{u}redi-Kalai (FFK)
inequalities (see \cite{frankl1988shadows}). In \cite{nevo2011gamma}, authors showed that the $\g$-vector
of the barycentric subdivision of a homology sphere satisfies the FFK inequalities, i.e., the $f$-vector of a balanced simplicial complex.\\
The first  aim of this paper  is the confirmation of  Conjecture \ref{NP} in the case of the interval subdivision of a homology sphere. The main theorem is stated as:
  \begin{teo}\label{main}
    If $\D$ is a simplicial complex with a nonnegative and symmetric $h$-vector, then the $\g$-vector of the interval subdivision of $\D$ is the $f$-vector of a balanced simplicial complex.
  \end{teo}
\noindent This work is based on  the study of certain refinement of
Eulerian numbers of type $B$ used by the present authors in
\cite{AN} to describe the $h$-vector of the  interval subdivision
$\Int(\Delta)$ of a simplicial complex $\Delta$. The  interval
subdivision of a simplicial complex was introduced by Walker in
\cite{walker1988canonical}.  By Walker
\cite[Theorem 6.1(a)]{walker1988canonical}, the simplicial complex
of all chains in the partially ordered set
$I(\Delta\setminus\emptyset):=\{[A, B]\ | \ \emptyset\neq A \subseteq B \in \Delta \}$ is a subdivision of $\Delta$.
In \cite{AN}, authors have given the combinatorial description of $f$- and $h$-vectors of a simplicial complex under the interval subdivision.\\
The second aim of this paper is to study the local $\g$-vector of
the interval subdivision of a simplex.  Stanley
\cite{stanley1992subdivisions} introduced the  local $h$-vector of a
topological subdivision of a  simplex as a tool to study the face
numbers of subdivisions of simplicial complexes. The local
$h$-vector plays an important role to answer the question posed by
Kalai and Stanley: {\em whether the $h$-vector increases
coordinate-wise after such subdivision of a Cohen-Macaulay complex?}
Since the local $h$-vector is symmetric it makes sense to define a
local $\g$-vector, which was introduced by Athanasiadis in
\cite{athanasiadis2012flag}. It already follows from
\cite{athanasiadis2013symmetric} that the local $\g$-vector of the
interval subdivision of a simplex is nonnegative. Here, we give
another proof of this result by answering a question asked by
Juhnke-Kubitzke et al in \cite{Juhnke-Kubitzke2018}.

\begin{ques}\label{ques} \cite[Problem 4.9]{Juhnke-Kubitzke2018}
  Find classes of subdivisions $\Gamma$ such that\linebreak $h(\Gamma_F , x)-h(\partial(\Gamma_F ), x)$ is
nonnegative, unimodal or $\g$-nonnegative. Moreover, for those
classes try to find a combinatorial interpretation of the
coefficients of \linebreak $h(\Gamma_F ,
x)-h(\partial(\Gamma_F ), x)$ respectively the coefficients of its
$\g$-polynomial.
\end{ques}
We show that for the  interval subdivision $\Gamma$ of a simplex,
$h(\Gamma_F , x)-h(\partial(\Gamma_F ), x)$ is  nonnegative and
$\g$-nonnegative. Along the way, we give a combinatorial
interpretation of the coefficients of  $h(\Gamma_F
, x)-h(\partial(\Gamma_F ), x)$ and as well of
the coefficients of its $\g$-polynomial. \\
%The enumeration data e.g., $f$, $h$, $\gamma$, $g$-vectors of barycentric subdivision of simplicial complexes has been extensively studied in the literature, see \cite{stanley1992subdivisions, brenti2008f, kubitzke2009lefschetz,  murai2010face, nevo2011gamma, petersen2015eulerian}.
% These results are important in the study of intrinsic behavior of interval partition and open a window for many new investigations.
We give a brief description on necessary definitions and notions in
Section $2$. In particular, we recall some known results about the
FFK-vectors and balanced complexes. In Section $3$, we give the
combinatorial foundation needed to  prove of the main result.
The proof of Theorem \ref{main} is given in Section $4$. In the
last section, we give a proof of nonnegativity of  the local $\g$-vector
of interval subdivision of a simplex. Additionally, we give a
geometrical description of the symmetric Eulerian polynomial
$B^{++}_n(x)$, defined in Section $3$.
 %We also show that if $h$-vector of the simplicial complex $\Delta$ satisfies the Dehn-Sommerville relations, so does the $h$-vector of the interval subdivision $\Int(\Delta)$. We investigate some properties of $f$- and $h$-vector transformation matrices. We show that these transformations are diagonalizable and similar. We also give a characterization of eigenvector corresponding to the highest eigenvalue of the $h$-vector transformation. In Section $4$, we prove the main result that if $\Delta$ has a nonnegative
%$h$-vector then the $h$-polynomial of its interval subdivision has only real zeros. We also  prove that the $j$-Eulerian polynomials of type $B^+$ and type $B^-$ are real-rooted.    %At the end, we  propose some open problems.
  %A very first question would be asked whether $h$-polynomial of the interval subdivision $\Int(\Delta)$ is real-rooted.
% We conjecture that  it has only real roots. Moreover, we illustrate  some computation for $\gamma$-vector and local $h$-vector of the interval subdivision.
\section{Background and Basic Notions}
A simplicial complex $\Delta$ on a finite vertex set $V$ is a
collection of subsets of $V,$ such that $\{v\}\in \Delta$ for all
$v\in V$, and if $F\in \Delta$ and $E\subseteq F$, then $E\in
\Delta$. The members of $\Delta$ are known as \textit{faces}. The
dimension of a face $F$ is $|F|-1$. Let $d = \max\{|F| : F \in \D\}$
and define the dimension of $\Delta$ to be $\dim \Delta = d-1$.\\

The \textit{$f$-polynomial} of a $(d-1)$-dimensional simplicial complex $\D$ is defined as:
$$f_{\D}(t)=\sum_{F\in \D}t^{\dim F+1}=\sum_{i=0}^{d}f_it^i,$$ where $f_i$ is the number of faces of dimension $i-1$. As $\dim \emptyset=-1$, so $f_0=1$. The sequence $f(\D)=(f_0,f_1,\ldots,f_d)$ is called the \textit{$f$-vector} of $\D$. Define the \textit{$h$-vector}
$h(\D) = (h_0, h_1, \ldots, h_d)$ of $\D$ by the \textit{$h$-polynomial}: $$h_{\D}(t):=(1-t)^df_{\D}(t/(1-t))=\sum_{i=0}^{d}h_it^i.$$
If $\D$ is a $(d-1)$-dimensional simplicial complex with the symmetric $h$-vector, i.e., the symmetric $h$-polynomial, then there exist integers $\g_i$ such that $$h_{\D}(t)=\sum_{0}^{\lfloor d/2\rfloor}\g_it^i(1+t)^{d-2i},$$ the sequence  $\g(\D)=(\g_0,\g_1,\ldots, \g_{\lfloor d/2 \rfloor})$ is called the $\g$-vector of $\D$.\\

The interval subdivision $\Int(\Delta)$ of a simplicial complex
$\Delta$ is the simplicial complex on the vertex set
$I(\Delta\setminus \emptyset)$, where $I(\Delta\setminus
\emptyset):=\{[A, B]\ | \ \emptyset\neq A \subseteq B \in \Delta \}$
as a partially ordered set ordered by inclusion defined as
$[A,B]\subseteq [A',B'] \in I(\Delta\setminus \emptyset) $ if and
only if $A'\subseteq A\subseteq B\subseteq B'$. By Walker
\cite{walker1988canonical}, the simplicial complex of all chains in
the partially ordered set $I(\Delta\setminus \emptyset)$ is a
subdivision of $\Delta$. In \cite{AN}, authors have given the combinatorial description of $f$- and $h$-vectors of a simplicial complex under the interval subdivision.

\begin{teo}\label{3.1}\cite[Theorem 2.2 and Theorem 3.1]{AN}
   Let $\Delta$ be a $(d-1)$-dimensional simplicial complex. Then the transformation of $f$-vector  of
$\Delta$ to $f$-vector of interval subdivision $\Int(\Delta)$ is given as
 $$f(\Int(\Delta))= [(\mathcal{F}_d)_{k,l}]_{0\leq k,l\leq d}f(\D),$$
 where
 $$(\mathcal{F}_d)_{0,l}=\left\{
               \begin{array}{ll}
                 1, & \hbox{$l=0$;} \\
                 0, & {l>0.}
               \end{array}
             \right.
$$ and for $1\leq k\leq d$, we have
\begin{equation}\label{matrix}
  (\mathcal{F}_d)_{k,l}=\sum_{j=0}^{k-1}(-1)^j{k-1 \choose
j}[(2k-2j)^l-(2k-2j-1)^l]
\end{equation} and the transformation of $h$-vector  of
$\Delta$ to $h$-vector of  $\Int(\Delta)$ is given as
\begin{equation}\label{h-vector}
   h(\Int(\D))=[B^+(d+1,s+1,r)]_{0\leq s,r\leq d}h(\D),
   \end{equation}
   where $B^+(d+1,s+1,r)$ is the number of $\si\in B_{d+1}$ with $\si_1=s$, $\si_{d+1}>0$ and $\des_B(\si)=r$.
\end{teo}

Our goal is to show that the $\g$-vector of the interval subdivision of $\D$ is $f$-vector of a balanced simplicial complex.
\subsection{FFK-vectors}
A simplicial complex $\D$ on the  vertex set $V$ is called \textit{$k$-colorable} if there is a
function $c : V \rightarrow \{1, 2, \ldots, k\}$, called a\textit{ coloring} of its vertices, such that
 every face has distinctly colored vertices. If a $(d-1)$-dimensional
complex $\D$ is $d$-colorable, we say it is a \textit{balanced}
simplicial complex.\\ There is an analogue of the
Kruskal-Katona-Sch\"{u}tzenberger(KKS) inequalities for
$k$-colorable simplicial complexes, due to Frankl, F\"{u}redi, and
Kalai \cite{frankl1988shadows}, known as FFK-inequalities.  The
vector satisfying FFK inequalities with respect to $k$ is called
$k$-\textit{FFK-vector}. It is shown in \cite{frankl1988shadows} that every face vector of a $k$-colorable complex is a $k$-FFK-vector and every $k$-FFK-vector is a face vector of some $k$-colorable complex. For more details, see
\cite{frankl1988shadows,nevo2011,nevo2011gamma}. Lets state some
definitions  from \cite{nevo2011gamma}.
\begin{defi}\label{good}\cite[Definition 3.6]{nevo2011gamma}
  Let $f=(1,f_1,\ldots,f_d)$ be the $f$-vector of a $(d-1)$-dimensional balanced complex.
  \begin{enumerate}
    \item \textbf{$(d+1)$-good.} Let $g = (g_1,\ldots, g_d, g_{d+1})$ be a sum of $d$-FFK-vectors, each of
which is dominated by $f$ . Some, but not all, of these $d$-FFK-vectors may be shorter than $f$; that
is, $g_{d+1}\neq 0$. Then we say that $(0, g)$ is $(d + 1)$-good for $f$. By \cite[Lemma 3.1]{nevo2011gamma}, it can be noted that $f + (0, g)$ is a $(d + 1)$-FFK-vector.
    \item \textbf{$d$-good.} Let $g = (g_1,\ldots, g_d) = f ^{(1)} +\cdots + f ^{(k)}$, with $g_d\neq 0$, be a sum of $(d-1)$-FFK-vectors such
that $f_i\geq (i +1) f^{(j)}_i$ for all $i$ and all $j$. Then we say that $(0, g)$ is $d$-good for $f$. By \cite[Lemma 3.5]{nevo2011gamma}, it can be noted that $f +(0, g)$
is a $d$-FFK-vector.
  \end{enumerate}
\end{defi}
 From the above definition, we have the following observation:
 \begin{obs}\label{3.7} \cite[Observation 3.7]{nevo2011gamma}
   Let $f = (1, f_1,\ldots, f_d)$ be the $f$-vector of a $(d-1)$-dimensional balanced complex. If $(0, g)$
is $d$-good for $f$ and $(0, g'
)$ is $(d+1)$-good for $f$, then $(0, g+g'
)$ is $(d+1)$-good for $f$.
 \end{obs}
 \subsection{Coxeter Complex of type $B$:} Here, we briefly discuss the simplicial complex
 whose $f$-vector is the $\g$-vector of Coxeter complex of type $B$
 (details can be found in \cite{stembridge2008coxeter, nevo2011, petersen2015eulerian}). Define the set of decorated permutations, $\mathrm{Dec}_n$, to be the set of all permutations
$\si \in S_n$ with bars in the left peak positions. The bars can come in
one of four styles: $\{|=|^0, |^1, |^2, |^3\}$, and thus for each $\si\in S_n$, we have $4^{\textrm{lpk}(\si)}$
decorated permutations. For example, here are three elements of $\Dec_7$:
$$4|237|^2651, 4|^1327|^3156, 25|^3137|^1654.$$
 Notice that the leftmost block is always increasing, even it is a singleton. Also the rightmost block may be strictly
 increasing.\\  Construct a balanced simplicial complex $\Gamma(\Dec_n)$ whose faces are elements of  $\Dec_n$ with
 $\dim \si=\textrm{lpk}(\si)-1$, where $\textrm{lpk}(\si):=\{1\leq i\leq n-1\ :\ \si_{i-1}<\si_i>\si_{i+1}\}$, with $\si_0=0$.
 A decorated permutation $\si$ covers $\tau$ if  $\tau$ can be obtained from $\si$ by  removing a bar  from  $B_i|B_{i+1}$ and reordering $B_iB_{i+1}$ as keep the
decreasing part of $B_i$ as it is, and rewrite $B_{i+1}$, together
with the increasing part of $B_i$, in increasing order. It is clear
that vertices are elements with only one bar. $\Gamma(\Dec_n)$ has
dimension $\lfloor \frac{n}{2}\rfloor-1$ and the color set of a face
$\si$ is $\mathrm{col}(\si)=\{\lceil i/2 \rceil \ :\
\si_i>\si_{i+1}\}$. By \cite[Corollary 4.5 (2)]{nevo2011}, we have
$$f(\Gamma(\Dec_n))=\g(B_n).$$
\begin{lema}\label{ineq}
  For all $n\geq 1$, $1\leq i \leq \lfloor \frac{n}{2}\rfloor, $ we have $$(i+1)\g_i(B_{n-1})\leq \g_i(B_{n}).$$
\end{lema}
\begin{proof}
  Let $\si=A_1| \cdots | A_{i+1}$ be any $(i-1)$-dimensional face of $\Gamma(\Dec_{n-1})$. Then a face of $\Gamma(\Dec_{n})$ can be otained from $\si$ by inserting $n$ to any of the $i+1$ block of $\si$. If we insert $n$ at the end of a block, no new left peak will be created and hence with this insertion, we get a $(i-1)$-dimensional face of $\Gamma(\Dec_{n})$ from a $(i-1)$-dimensional face of $\Gamma(\Dec_{n-1})$. Moreover, different faces in $\Gamma(\Dec_{n-1})$ give arise the disjoint set of faces in $\Gamma(\Dec_{n})$. Hence, the inequality follows.
\end{proof}

The next lemma follows from  Definition \ref{good} and Lemma \ref{ineq}:

\begin{lema}\label{5.5}
  Let $n=2d+1$. Then
   \begin{itemize}
            \item If $(0,f)$ is $d$-good for $\g(B_{n-2})$, then it is also $d$-good for $\g(B_{n-1})$.
                       \item If $(0,f)$ is $(d+1)$-good for $\g(B_{n-1})$, then it is also $(d+1)$-good for $\g(B_{n})$.
                     \end{itemize}
\end{lema}
\begin{proof}
  The proof is similar to the proof of  Lemma 5.5 \cite{nevo2011gamma}.
\end{proof}

\subsection{Subdivisions and Local $h$-Vectors:} The notion of local $h$-vectors was firstly studied by Stanley \cite{stanley1992subdivisions}.
A \textbf{topological subdivision} of a simplicial complex $\D$ is a simplicial complex $\D'$ with a map $\theta: \D'\rightarrow \D$ such that,
 for any face $F\in \D$, the following holds: (a) $\D'_F:=\theta^{-1}(2^F)$ is a subcomplex of $\D'$ which is homeomorphic to a ball of dimension $\dim(F)$; (b) the interior of $\D'_F$ is equal to $\theta^{-1}(F)$.  The face $\theta(G)\in \D$ is called  the \textbf{carrier} of  $G\in \D'$. The subdivision $\D'$ is called \textbf{quasi-geometric } if no face of $\D'$ has the carriers of its vertices contained in a face of $\D$ of smaller dimension. Moreover,
$\D'$ is called \textbf{geometric}  if there exists a geometric realization of $\D'$
which geometrically subdivides a geometric realization of  $\D$, in the way prescribed by $\theta$.
Clearly, all geometric subdivisions (such as the interval subdivisions considered in this paper) are quasi-geometric.
For more detail, we refer to \cite{stanley1992subdivisions} and a survey by Athanasiadis \cite{athanasiadis2016survey}. \\
Let $V$ be non-empty finite set. Let $\Gamma$ be a subdivision of a
$(d-1)$-dimensional simplex $2^V$. Then
\begin{equation}\label{local h-poly}
  \ell_V(\Gamma,x)=\sum_{F\subseteq V}(-1)^{d-|F|}h(\Gamma_F,x)=\sum_{j=0}^{d}\ell_jx^j
\end{equation}
is known as the \textbf{local $h$-polynomial} of $\Gamma$ (with respect to $V$). The sequence $\ell(\Gamma)=(\ell_0,\ldots,\ell_d)$ is the \textbf{local $h$-vector } of $\Gamma$( with respect to $V$). By Stanley \cite{stanley1992subdivisions}, it is well-known that the local $h$-vector $\ell_V(\Gamma,x)$ is symmetric. Thus it can be uniquely written in the following way: $$\ell_V(\Gamma,x)=\sum_{i=0}^{\lfloor d/2 \rfloor}\xi_ix^i(1+x)^{d-2i},$$
where $\xi_i\in \mathbb{Z}$. The sequence $\xi{(\Gamma)}=(\xi_0,\ldots,\xi_{\lfloor d/2 \rfloor})$ is called the local $\g$-vector of $\Gamma$(w.r.t. V). It is natural to ask whether the local $\g$-vector is nonnegative or not.  The local $\g$-vector is nonnegative  for
special classes of subdivisions including barycentric, edgewise, cluster and interval subdivisions
of a simplex \cite{athanasiadis2012flag, athanasiadis2012local, athanasiadis2013symmetric}.  The local $h$-polynomial of the interval subdivision of a simplex $2^V$ has nonnegative, symmetric and unimodal coefficients and has a nice combinatorial description due to Athanasiadis and Savvidou \cite{athanasiadis2013symmetric}.
\begin{teo}\label{local gamma}\cite[Corollary 1.2, Theorem 1.3, Propostion 4.1]{athanasiadis2013symmetric}
Let $\Gamma$ be the interval subdivision of a $(d-1)$-dimensional simplex $2^V$.
  The local $h$-polynomial $\ell_V(\Gamma,x)$ of $\Gamma$ is nonnegative, symmetric and $\g$-nonnegative. More precisely,
  \begin{equation}\label{local h}
    \ell_V(\Gamma,x)=\sum_{\si\in \mathcal{D}_d^B\cap B^*_d}x^{\mathrm{exc}{(\si)}},
  \end{equation}
  where $\mathrm{exc}_B(\si)$ is the number of $B$-excedances of $\si$; $\mathcal{D}_d^B$ is the set of derangements of $B_d$ and $B^*_d:=\{\si\in B_d\ :\ \si(m_{\si})>0\}$, $m_{\si}$ is the minimal element of $\{\si_1,\ldots,\si_d\}$.
\end{teo}
Here, we give an other proof of the nonnegativity of local $\g$-vector of the interval subdivision of a simplex using the result of Juhnke-Kubitzke et al \cite{Juhnke-Kubitzke2018}. They gave an expression of the local $h$-vector which involves differences of
$h$-vectors of restrictions of the subdivision and their boundary as well as derangement polynomials.
\begin{teo}\label{Juhnke}\cite[Theorem 4.4]{Juhnke-Kubitzke2018}
 Let $\Gamma$ be a subdivision  of a simplex $2^V$. The local $h$-polynomial of $\Gamma$ can be written as
 $$\ell_V(\Gamma,x)=\sum_{F\subseteq V}[h(\Gamma_F,x)-h(\partial(\Gamma_F),x)] . \mathfrak{d}^A_{|V|\setminus|F|}(x),$$
 where  $\mathfrak{d}^A_{d}(x)$ is the usual  derangement polynomial of order $d$. In particular when $\Gamma=\Int(2^{V})$, we obtain
\begin{equation}\label{ell}
\ell_V(\Int(2^{V}),x)=\sum_{F\subseteq V}[h(\Int(2^{F}),x)-h(\partial(\Int(2^{F})),x)] . \mathfrak{d}^A_{|V|\setminus|F|}(x)
\end{equation}
\end{teo}

In Section \ref{last}, we  answer  the Question \ref{ques} in case of the interval subdivision by giving a nice combinatorial description on differences of
$h$-vectors of restrictions of the subdivision and their boundary.

\section{Symmetric  Eulerian Polynomials of type $B$}
 Let $B_n$  be  the group consisting of
all the bijections $\sigma$ of the set $ \{\pm 1,\ldots,\pm n\}$ onto itself such that $\sigma_{-i}= -\sigma_{i}$ for all $i\in \{\pm 1,\ldots, \pm n\}$.
For $\sigma\in B_n$, the \textit{Descent set} is defined as  $$\textrm{Des}_B(\sigma):=\{i\in [0,n-1] : \sigma_i>\sigma_{i+1}\},$$ where $\sigma_0=0$ and the type $B$-\textit{descent number} is defined  as $\textrm{des}_B(\sigma):=|\hbox{Des}_B(\sigma)|$.\\ We use the notation $\bar{s}:=-s$. Let $$B_n^{+}:=\{\sigma\in B_n : \sigma_{n}>0\}$$ and for $\bar{n}\leq j\leq n$,
$$B_{n,j}^{+}:=\{\sigma\in B_n^+ : \sigma_1=j\}.  $$ Denote $B^+(n,j,k)$
 by the number of elements in $B_{n,j}^{+}$ with exactly $k$ descents. Now, define the
  \textbf{\textit{$j$-Eulerian polynomials of type $B^+$}} by
  \begin{equation}\label{B+}
    B^+_{n,j}(t):=\sum_{\sigma\in B^+_{n,j}}t^{\des_B(\sigma)}=\sum_{k=0}^{n-1}B^+(n,j,k)t^k,
  \end{equation}

  Note that the usual Eulerian polynomial of type $B^+$ is the descent generating function for all of $B^+_n:$
  $$B^+_n(t)=\sum_{\si \in B^+_n}t^{\des_B(\si)}=\sum_{j=\bar{n}}^{n}B^+_{n,j}.$$
Lets denote $$B^{++}_n(x):=\sum_{j=1}^{n}B^+_{n,j}=\sum_{k=0}^{n}B^{++}(n,k)x^k,$$ where $B^{++}(n,k)$ is the number of signed permutations of $B^{++}_n:=\{\si\in B^+_n\ :\ \si_1>0\}$ with exactly $k$ descents  and $$B^{-+}_n(x):=\sum_{j=1}^{n}B^+_{n,\overline{j}}=\sum_{k=0}^{n}B^{-+}(n,k)x^k,$$ where $B^{-+}(n,k)$ is the number of signed permutations of $B^{-+}_n:=\{\si\in B^+_n\ :\ \si_1<0\}$ with exactly $k$ descents.\\

Lets state the recurrence relations:

\begin{lema}\label{Ob1}\cite[Lemma 3.4]{AN}
For $1\leq s\leq n$ and $1\leq r\leq  n-1$, we have the following relations:
\begin{enumerate}
 \item $B^+(n,s,r)=B^+(n,n-s+1,n-r-1)$ and thus $$B^+_{n,s}(t)=t^{n-1}B^+_{n,n-s+1}(t).$$
 \item $B^+(n,\overline {s},r)=B^+(n,\overline{n-s+1},n-r-1)$ and thus $$B^+_{n,\bar{s}}(t)=t^{n-1}B^+_{n,\overline{n-s+1}}(t).$$
    \item $$B^+(n,s,r)=\sum_{j=1}^{s-1}B^+(n-1,j,r-1)+\sum_{j=1}^{n-1}B^+(n-1,\bar{j},r)+\sum_{j=s}^{n-1}B^+(n-1,j,r).$$ Thus, the recurrence  relation holds:
    $$B^+_{n,s}(t)=t\sum_{j=1}^{s-1}B^+_{n-1,j}(t)+\sum_{j=s}^{n-1}B^+_{n-1,j}(t)+\sum_{j=1}^{n-1}B^+_{n-1,\overline{j}}(t),$$ with initial conditions $B^+_{1,1}(t)=1$ and $B^+_{1,\overline{1}}(t)=0$.
    \item  $$B^+(n,\overline{s},r)=\sum_{j=1}^{n-1}B^+(n-1,j,r-1)+\sum_{j=s}^{n-1}B^+(n-1,\overline{j},r-1)+\sum_{j=1}^{s-1}B^+(n-1,\overline{j},r).$$ Thus, the recurrence  relation holds:
    $$B^+_{n,\overline{s}}(t)=t\sum_{j=1}^{n-1}B^+_{n-1,j}(t)+t\sum_{j=s}^{n-1}B^+_{n-1,\overline{j}}(t)+\sum_{j=1}^{s-1}B^+_{n-1,\overline{j}}(t).$$

\end{enumerate}
\end{lema}
Using Lemma \ref{Ob1} and bijection between $B^+_n$ and $B^-_n:=\{\si\in B_n \ :\ \si_n<0 \}$ by mapping $\si=(\si_1,\dots,\si_n)$ to $\overline{\si}=(\overline{\si_1},\ldots,\overline{\si_n})$, we have the following relation:
\begin{obs}\label{Obs}
  We have $B^+_{n,1}(t)=B^+_{n-1}(t)$ and $B^+_{n,n}(t)= t^{n-1}B^+_{n-1}(t^{-1})=B^-_{n-1}(t),$ by \cite[Lemma 3.4]{AN}
  where $B^-_{n}(t)$ is the descent generating function for all of $B^-_n$. Thus we get that $B_n(t)=B^+_{n,1}(t)+B^+_{n,n}(t)$.
\end{obs}

Since the $j$- Eulerian polynomials are not symmetric so define the symmteric $j$-Eulerian polynomials as:
$$\mathbb{B}^+_{n,j}(t)=\sum_{\si\in B^+_{j}\cup B^+_{n-j+1}}t^{\des_B(\si)}$$ and
$$\mathbb{B}^+_{n,\overline{j}}(t)=\sum_{\si\in B^+_{\overline{j}}\cup B^+_{\overline{n-j+1}}}t^{\des_B(\si)}$$
Observe that

                       $$\BB(t)= \left\{ \begin{array}{ll}
                        \B(t)+B^+_{n,n-j+1}(t), & \hbox{$j\neq (n+1)/2$;} \\
                        \B(t), & \hbox{$j=(n+1)/2$.}
                       \end{array}
                     \right.$$  and
                       $$\mathbb{B}^+_{n,\overline{j}}(t)=\left\{
                       \begin{array}{ll}
                         {B}^+_{n,\overline{j}}(t)+{B}^+_{n,\overline{n-j+1}}(t), & \hbox{$j\neq (n+1)/2$;} \\
                          {B}^+_{n,\overline{j}}(t), & \hbox{$j=(n+1)/2$.}
                        \end{array}
                      \right.$$
By Lemma \ref{Ob1}(1), the polynomials $\BB(t)$ have symmetric
coefficients, and hence a $\gamma$-vector exists. Clearly, $\BB(t)$
has symmetry axis at degree $\lfloor \frac{n-1}{2}\rfloor$. If
$$\BB(t)=\sum_{k=0}^{\lfloor \frac{n-1}{2}\rfloor}
\gamma_k^{(n,j)}t^k(1+t)^{n-1-2k},$$ let
$\gamma^{(n,j)}=(\gamma_0^{(n,j)},\gamma_1^{(n,j)},\ldots,
\gamma_{\lfloor \frac{n-1}{2}\rfloor}^{(n,j)})$ denotes the
corresponding gamma vector.\\
Lets define the following polynomials for $1\leq j< (n+1)/2$
 $$\BBB(t)=t\B(t)+\Bj(t)$$ and  $$\widetilde{\mathbb{B}}^+_{n,\overline{j}}(t)=tB^+_{n,\overline{n-j+1}}(t)+B^+_{n,\overline{j}}(t).$$
 The above polynomials are also symmetric by Lemma \ref{Ob1} and have symmetry axis at $\lfloor n/2 \rfloor$. Let
$\ga^{(n,j)}=(\ga_0^{(n,j)},\ga_1^{(n,j)},\ldots, \ga_{\lfloor \frac{n}{2}\rfloor}^{(n,j)})$  and $\ga^{(n,\overline{j})}=(\ga_0^{(n,\overline{j})},\ga_1^{(n,\overline{j})},\ldots, \ga_{\lfloor \frac{n}{2}\rfloor}^{(n,\overline{j})})$  denote the $\g$-vectors for $\BBB(t)$ and $\widetilde{\mathbb{B}}^+_{n,\overline{j}}(t)$ respectively.\\

 It can be observed that $B^{++}_n(t)$ and $B^{-+}_n(t)$ are symmetric polynomials by Lemma \ref{Ob1}, i.e., we have
 \begin{equation}\label{sym}
   B^{++}(n,k)=B^{++}(n,n-k-1) \hbox{\ and\ }  B^{-+}(n,k)=B^{-+}(n,n-k-1)
 \end{equation}
To prove Theorem \ref{main}, we need the following lemma:
\begin{lema}\label{Re}
  The following recurrence relations hold for $\g^{(n,j)}$ and $\ga^{(n,j)}$:

  \begin{description}
    \item[1] If $j=(n+1)/2$, then  $$\g^{(n,(n+1)/2)}=\sum_{k=1}^{(n-1)/2}[\ga^{(n-1,k)}+\g^{(n-1,\overline{k})}].$$
    \item[2]  If $j=\overline{(n+1)/2}$, then  $$\g^{(n,\overline{(n+1)/2)}}=\sum_{k=1}^{(n-1)/2}[\ga^{(n-1,\overline{k})}+(0,\g^{(n-1,k)})].$$
    \item[3]   For $1\leq j<(n+1)/2$, $$\g^{(n,j)}=2\sum_{k=1}^{\lfloor n/2\rfloor}\g^{(n-1,\overline{k})}+2\sum_{k=1}^{j-1}\ga^{(n-1,k)}+\sum_{k=j}^{\lfloor n/2\rfloor}\g^{(n-1,{k})}.$$
    \item[4]  For $1\leq j<(n+1)/2$,
    $$\g^{(n,\overline{j})}=2\sum_{k=1}^{\lfloor n/2\rfloor}(0,\g^{(n-1,k)})+2\sum_{k=1}^{j-1}\ga^{(n-1,\overline{k})}+\sum_{k=j}^{\lfloor n/2\rfloor}\g^{(n-1,\overline{k})}.$$
    \item[5]  For $1\leq j<(n+1)/2$, $$\ga^{(n,j)}=\sum_{k=1}^{j-1}\ga^{(n-1,k)}+2\sum_{k=j}^{\lfloor n/2\rfloor}(0,\g^{(n-1,{k})})+\sum_{k=1}^{\lfloor n/2\rfloor}\g^{(n-1,\overline{k})}.$$
    \item[6] For $1\leq j<(n+1)/2$,
    $$\ga^{(n,\overline{j})}=\sum_{k=1}^{j-1}\ga^{(n-1,\overline{k})}+2\sum_{k=j}^{\lfloor n/2\rfloor}(0,\g^{(n-1,\overline{k})})+\sum_{k=1}^{\lfloor n/2\rfloor}(0,\g^{(n-1,{k})}).$$
  \end{description}
\end{lema}
\begin{proof}
For $j=(n+1)/2$, by Lemma \ref{Ob1} (3)
\begin{align*}
 \mathbb{B}^+_{n,(n+1)/2}(t)=&B^+_{n,(n+1)/2}(t)\\
 = &  t\sum_{k=1}^{(n-1)/2}B^+_{n-1,k}(t)+\sum_{k=(n+1)/2}^{n-1}B^+_{n-1,k}(t)+\sum_{k=1}^{n-1}B^+_{n-1,\overline{k}}(t)\\
  = &  t\sum_{k=1}^{(n-1)/2}B^+_{n-1,k}(t)+\sum_{k=1}^{(n-1)/2}B^+_{n-1,n-k}(t)+
  \sum_{k=1}^{(n-1)/2}B^+_{n-1,\overline{k}}(t)+\sum_{k=(n+1)/2}^{n-1}B^+_{n-1,\overline{k}}(t)\\
  = & \sum_{k=1}^{(n-1)/2}(tB^+_{n-1,k}(t)+B^+_{n-1,n-k}(t))+\sum_{k=1}^{(n-1)/2}(B^+_{n-1,\overline{k}}(t)+B^+_{n-1,\overline{n-k}}(t))\\
  = & \sum_{k=1}^{(n-1)/2}(\widetilde{\mathbb{B}}^+_{n-1,k}(t)+\mathbb{B}^+_{n-1,\overline{k}}(t)),
\end{align*}
%which proves {1:}.\\
For $j=\overline{(n+1)/2}$, by Lemma \ref{Ob1} (3)
\begin{align*}
 \mathbb{B}^+_{n,\overline{(n+1)/2}}(t)=&B^+_{n,\overline{(n+1)/2}}(t)\\
 = &  t\sum_{k=1}^{n-1}B^+_{n-1,k}(t)+t\sum_{k=(n+1)/2}^{n-1}B^+_{n-1,\overline{k}}(t)+\sum_{k=1}^{(n-1)/2}B^+_{n-1,\overline{k}}(t)\\
  = &  t\sum_{k=1}^{(n-1)/2}B^+_{n-1,k}(t)+t\sum_{k=1}^{(n-1)/2}B^+_{n-1,n-k}(t)+t\sum_{k=1}^{(n-1)/2}B^+_{n-1,\overline{n-k}}(t)
  +\sum_{k=1}^{(n-1)/2}B^+_{n-1,\overline{k}}(t)\\
  = &t\sum_{k=1}^{(n-1)/2}(B^+_{n-1,k}(t)+B^+_{n-1,n-k}(t))+\sum_{k=1}^{(n-1)/2}(tB^+_{n-1,\overline{n-k}}(t)+B^+_{n-1,\overline{k}}(t))\\
  = & \sum_{k=1}^{(n-1)/2}(t{\mathbb{B}}^+_{n-1,k}(t)+\widetilde{\mathbb{B}}^+_{n-1,\overline{k}}(t)).
\end{align*}
To prove 3: if $1\leq j<(n+1)/2$,
 \begin{align*}
\mathbb{B}^+_{n,j}(t)=&B^+_{n,j}(t)+B^+_{n,n-j+1}(t)\\
 = &  t\sum_{k=1}^{j-1}B^+_{n-1,k}(t)+\sum_{k=j}^{n-1}B^+_{n-1,{k}}(t)+t\sum_{k=1}^{n-j}B^+_{n-1,k}(t)
  +\sum_{k=n-j+1}^{n-1}B^+_{n-1,{k}}(t)+2\sum_{k=1}^{n-1}B^+_{n-1,\overline{k}}(t)\\
   = &  2t\sum_{k=1}^{j-1}B^+_{n-1,k}(t)+(1+t)\sum_{k=j}^{n-j}B^+_{n-1,n-k}(t)+2\sum_{k=n-j+1}^{n-1}B^+_{n-1,k}(t)+2\sum_{k=1}^{n-1}B^+_{n-1,\overline{k}}(t)
  \end{align*}
 \begin{align*}
 = & 2\sum_{k=1}^{j-1}(tB^+_{n-1,k}(t)+B^+_{n-1,n-k}(t))+(1+t)\sum_{k=j}^{n-j}B^+_{n-1,n-k}(t)+2\sum_{k=1}^{n-1}B^+_{n-1,\overline{k}}(t)\\
  = & 2\sum_{k=1}^{j-1}\widetilde{\mathbb{B}}^+_{n-1,k}(t)+(1+t)\sum_{k=j}^{\lfloor n/2\rfloor}\mathbb{B}^+_{n-1,k}(t)+2\sum_{k=1}^{\lfloor n/2\rfloor}\mathbb{B}^+_{n-1,\overline{k}}(t).
\end{align*}
To prove 4: if $1\leq j<(n+1)/2$,

\begin{align*}
 \mathbb{B}^+_{n,\overline{j}}(t)=&B^+_{n,\overline{j}}(t)+B^+_{n,\overline{n-j+1}}(t)\\
 = &  t\sum_{k=1}^{n-1}B^+_{n-1,k}(t)+t\sum_{k=j}^{n-1}B^+_{n-1,\overline{k}}(t)+\sum_{k=1}^{j-1}B^+_{n-1,\overline{k}}(t)+t\sum_{k=1}^{n-1}B^+_{n-1,k}(t) +t\sum_{k=n-j+1}^{n-1}B^+_{n-1,{k}}(t)\\
 &+\sum_{k=1}^{n-j}B^+_{n-1,\overline{k}}(t)\\
 %  \end{align*}
%\begin{align*}
  = &  2t\sum_{k=1}^{n-1}B^+_{n-1,k}(t)+t\sum_{k=1}^{n-j}B^+_{n-1,\overline{n-k}}(t)+t\sum_{k=1}^{j-1}B^+_{n-1,\overline{n-k}}(t)
   +\sum_{k=1}^{j-1}B^+_{n-1,\overline{k}}(t)\\
   &+\sum_{k=1}^{n-j}B^+_{n-1,\overline{k}}(t)\\
  = & 2t\sum_{k=1}^{n-1}B^+_{n-1,k}(t)+2\sum_{k=1}^{j-1}(tB^+_{n-1,\overline{n-k}}(t)+B^+_{n-1,\overline{k}}(t))+ t\sum_{k=j}^{n-j}B^+_{n-1,\overline{n-k}}(t)\\
  &+\sum_{=j}^{n-j} B^+_{n-1,\overline{k}}(t)\\
  = & 2t\sum_{k=1}^{\lfloor n/2\rfloor}\mathbb{B}^+_{n-1,{k}}(t)+2\sum_{k=1}^{j-1}\widetilde{\mathbb{B}}^+_{n-1,\overline{k}}(t)+(1+t)\sum_{k=j}^{\lfloor n/2\rfloor}{\mathbb{B}}^+_{n-1,\overline{k}}(t).
\end{align*}
To prove 5: if $1\leq j<(n+1)/2$,
\begin{align*}
 \widetilde{\mathbb{B}}^+_{n,j}(t)=&tB^+_{n,j}(t)+B^+_{n,n-j+1}(t)\\
 = &  t^2\sum_{k=1}^{j-1}B^+_{n-1,k}(t)+t\sum_{k=j}^{n-1}B^+_{n-1,{k}}(t)+t\sum_{k=1}^{n-1}B^+_{n-1,\overline{k}}(t)+ t\sum_{k=1}^{n-j}B^+_{n-1,k}(t)+\sum_{k=n-j+1}^{n-1}B^+_{n-1,{k}}(t)\\
 &+\sum_{k=1}^{n-1}B^+_{n-1,\overline{k}}(t)\\
 = &  t^2\sum_{k=1}^{j-1}B^+_{n-1,k}(t)+t\sum_{k=1}^{n-j}B^+_{n-1,n-k}(t)
 \end{align*}
 \begin{align*}
    &+t\sum_{k=1}^{n-j}B^+_{n-1,k}(t)+ \sum_{k=1}^{j-1}B^+_{n-1,n-k}(t)+(1+t)\sum_{k=1}^{n-1}B^+_{n-1,\overline{k}}(t)\\
 = & (1+t)\sum_{k=1}^{j-1}(tB^+_{n-1,k}(t)+B^+_{n-1,n-k}(t))+2t\sum_{k=j}^{n-j}B^+_{n-1,n-k}(t)+(1+t)\sum_{k=1}^{n-1}B^+_{n-1,\overline{k}}(t)\\
  = &(1+ t)\sum_{k=1}^{j-1}\widetilde{\mathbb{B}}^+_{n-1,k}(t)+2t\sum_{k=j}^{\lfloor n/2\rfloor}\mathbb{B}^+_{n-1,n-k}(t)+(1+t)\sum_{k=1}^{\lfloor n/2\rfloor}\mathbb{B}^+_{n-1,\overline{k}}(t).
\end{align*}
\\

To prove 6: if $1\geq j<{(n+1)/2}$,
\begin{align*}
 \widetilde{\mathbb{B}}^+_{n,\overline{j}}(t)=&tB^+_{n,\overline{n-j+1}}(t)+B^+_{n,\overline{j}}(t)\\
 = &  t^2\sum_{k=1}^{n-1}B^+_{n-1,k}(t)+t^2\sum_{k=n-j+1}^{n-1}B^+_{n-1,\overline{k}}(t)+t\sum_{k=1}^{n-j}B^+_{n-1,\overline{k}}(t)+t\sum_{k=1}^{n-1}B^+_{n-1,k}(t)\\
% \end{align*}
% \begin{align*}
 &+t\sum_{k=j}^{n-1}B^+_{n-1,\overline{k}}(t)+\sum_{k=1}^{j-1}B^+_{n-1,\overline{k}}(t)\\
  = &  t(1+t)\sum_{k=1}^{n-1}B^+_{n-1,k}(t)+t^2\sum_{k=1}^{j-1}B^+_{n-1,\overline{n-k}}(t)+\\
  & +t\sum_{k=1}^{n-j}B^+_{n-1,\overline{k}}(t)+ t\sum_{k=1}^{n-j}B^+_{n-1,\overline{n-k}}(t) +\sum_{k=1}^{j-1}B^+_{n-1,\overline{k}}(t)\\
  = & t(1+t)\sum_{k=1}^{\lfloor n/2\rfloor}B^+_{n-1,k}(t)+2t\sum_{k=j}^{n-j}B^+_{n-1,\overline{k}}(t)+(1+t)\sum_{k=1}^{j-1}(tB^+_{n-1,\overline{n-k}}(t)+B^+_{n-1,\overline{k}}(t))\\
  = &t(1+ t)\sum_{k=1}^{\lfloor n/2\rfloor}{\mathbb{B}}^+_{n-1,{k}}(t)+2t\sum_{k=j}^{\lfloor n/2\rfloor}\mathbb{B}^+_{n-1,\overline{k}}(t)+ (1+t)\sum_{k=1}^{j-1}\widetilde{\mathbb{B}}^+_{n-1,\overline{k}}(t).
\end{align*}
\end{proof}\section{$\g$-Vector}
If $\D$ has symmetric $h$-vector then by \cite[Theorem 3.1, Lemma 3.5]{AN}, $h(\Int(\D))$ is also symmetric. The following proposition holds:
\begin{prop}
  If $\Delta$ is a simplicial complex of dimension $n-1$ with symmetric $h$-vector $h(\Delta)=(h_0,h_1,\ldots,h_n)$, then
  \begin{equation*}
    h_r(\Int(\Delta))=\sum_{j=0}^{\lfloor n/2\rfloor}(B^+(n+1,j+1,r)+B^+(n+1,n-j,r))h_j,
  \end{equation*}

  and thus
  \begin{equation*}
     h(\Int(\Delta),t)=\sum_{j=0}^{\lfloor n/2\rfloor}h_j\mathbb{B}^+_{n+1,j+1}(t).
  \end{equation*}

  In terms of $\g$-vectors,
  \begin{equation*}
    \g(\Int(\Delta))=\sum_{j=0}^{\lfloor n/2\rfloor}h_j \g^{(n+1,j+1)}.
  \end{equation*}

\end{prop}

\begin{ej}
  If $n=5$ and $h(\Delta)=(h_0,h_1,h_2,h_3=h_2,h_4=h_1,h_5=h_0)$, then
  \begin{align*}
h(\Int(\Delta))^t=&\left(
                        \begin{array}{c}
                          h_0 \\
                          237h_0+192h_1+168h_2 \\
                          1682h_0+1728h_1+1752h_2 \\
                           1682h_0+1728h_1+1752h_2 \\
                           237h_0+192h_1+168h_2 \\
                          h_0 \\
                        \end{array}
                      \right)\\
                      =&h_0\left(
                                   \begin{array}{c}
                                     1 \\
                                     237\\
                                     1682 \\
                                     1682 \\
                                     237\\
                                     1 \\
                                   \end{array}
                                 \right)+h_1\left(
                                              \begin{array}{c}
                                                0 \\
                                                192 \\
                                                1728 \\
                                                1728 \\
                                                192 \\
                                                0 \\
                                              \end{array}
                                            \right)+h_1\left(
                                                         \begin{array}{c}
                                                           0 \\
                                                           168 \\
                                                           1752 \\
                                                           1752 \\
                                                           168 \\
                                                           0 \\
                                                         \end{array}
                                                       \right)
 \end{align*}
  Thus, $$h(\Int{\Delta},t)=h_0\mathbb{B}^+_{6,1}(t)+h_1\mathbb{B}^+_{6,2}(t)+h_3\mathbb{B}^+_{6,3}(t),$$ and
  $$\g(\Int{\Delta})=h_0\g^{(6,1)}+h_1\g^{(6,2)}+h_2\g^{(6,3)},$$ where $\g^{(6,1)}=(1,232,976),\g^{(6,2)}=(0,192,152)$ and $\g^{(6,3)}=(0,168,1248)$.
\end{ej}

\subsection{Proof of Theorem \ref{main}:} In this subsection,  we prove the main theorem that  $\g(\Int(\Delta))$ is an FFK-vector, i.e., an $f$-vector of a balanced complex.\\
Since $\mathbb{B}^+_{n,1}(t)={B}^+_{n,1}(t)+{B}^+_{n,n}(t)=B_{n-1}(t)$, where $B_{n-1}(t)$ is the Eulerian polynomial of type $B$. By \cite[Theorem 6.1(1)]{nevo2011}, the $\g$-vector  $\g(B_{n-1})$ of $B_{n-1}(t)$ is an FFK-vector, and thus $\g(B_{n-1})=\g^{(n,1)}$ is an FFK-vector for any $n$. Since $B_{n-1}(t)$ has symmetry at degree $\lfloor \frac{n-1}{2}\rfloor$, $\g^{(n,1)}=(1,f_1,\ldots, f_d)$, where $d=\lfloor \frac{n-1}{2}\rfloor$. Since $h_0=1$ for all simplicial complexes $\Delta$, $$\g(\Int(\Delta))=\g^{(n+1,1)}+h_1\g^{(n+1,2)}+\cdots,$$ where $n=\dim \Delta +1.$  Also note that $\g_0^{(n+1,j)}=0$ for all $j\neq1$. We will show that $\g^{(n,i)}$ is $d$-or $(d+1)$-good for $\g^{(n,1)}$. More precisely, we prove the following:

\begin{prop}\label{mainprop}
  Let $\g^{(n,1)}=(1,f_1,\ldots,f_d)$, where $d=\lfloor \frac{n-1}{2}\rfloor$.
  \begin{enumerate}
  \item If $n$ is odd, i.e., $n=2d+1$, then
    \begin{description}
      \item[a] $\g^{(n,j)}$, $1< j\leq (n+1)/2$, is $d$-good for $\g^{(n,1)}$,
      \item[b] $\g^{(n,\overline{j})}$, $1\leq j\leq (n+1)/2$, is $d$-good for $\g^{(n,1)}$,
       \item[c] $\ga^{(n,j)}$, $1\leq j < (n+1)/2$, is $d$-good for $\g^{(n,1)}$ and
       \item[d] $\ga^{(n,\overline{j})}$, $1\leq j< (n+1)/2$, is $(d+1)$-good for $\g^{(n,1)}$.
    \end{description}
    \item If $n$ is even, i.e., $n=2d+2$, then
    \begin{description}
      \item[a] $\g^{(n,j)}$, $1< j\leq n/2$, is $d$-good for $\g^{(n,1)}$,
      \item[b] $\g^{(n,\overline{j})}$, $1\leq j\leq n/2$, is $(d+1)$-good for $\g^{(n,1)}$,
      \item[c] $\ga^{(n,j)}$, $1\leq j \leq n/2$, is $(d+1)$-good for $\g^{(n,1)}$,
      \item[d] $\ga^{(n,\overline{j})}$, $1\leq j\leq  n/2$, is $(d+1)$-good for $\g^{(n,1)}$.
    \end{description}

  \end{enumerate}
\end{prop}
\begin{proof}
We will prove by induction on $d=\lfloor \frac{n-1}{2} \rfloor$. If $d=0$, we have

$$\g^{(1,1)}=1$$
$$\begin{array}{cc}
  \g^{(2,1)}=(1) & \widetilde{\g}^{(2,1)}=(0,2) \\
   g^{(2,\overline{1})}=(0,2) & \widetilde{\g}^{(2,\overline{1})}=(0,1)
\end{array}$$

It is trivial to see that $(0,2)$ and $(0,1)$ are $1$-good for $(1)$.  Thus, the claims hold for $d=0$. For $d=1$,

$$\begin{array}{cccc}
  \g^{(3,1)}=(1,4) & \g^{(3,2)}=(0,4) &  \widetilde{\g}^{(3,1)}=(0,4) & \\%{\widetilde{\g}^{(3,2)}=(0,4)}
  \g^{(3,\overline{1})}=(0,4) &  \g^{(3,\overline{2})}=(0,2) & \widetilde{\g}^{(3,\overline{1})}=(0,1,4) &
\end{array}$$

$$\begin{array}{cccc}
  \g^{(4,1)}=(1,20) & \g^{(4,2)}=(0,24) &  \widetilde{\g}^{(4,1)}=(0,8,16) & \widetilde{\g}^{(4,2)}=(0,10,8) \\
  \g^{(4,\overline{1})}=(0,8,16) &  \g^{(4,\overline{2})}=(0,6,24) & \widetilde{\g}^{(4,\overline{1})}=(0,1,20) & \widetilde{\g}^{(4,\overline{2})}=(0,2,16)
\end{array}$$

%$$\begin{array}{ccccc}
%  \g^{(5,1)}=(1,72,80) & \g^{(5,2)}=(0,68,112) & \g^{(5,3)}=(0,32,64) &  \widetilde{\g}^{(5,1)}=(0,16,128) & \widetilde{\g}^{(5,2)}=(0,22,104) \\
%  \g^{(5,\overline{1})}=(0,16,128) &  \g^{(5,\overline{2})}=(0,24, 144) &  \g^{(5,\overline{3})}=(0,30,168)& \widetilde{\g}^{(5,\overline{1})}=(0,1,72,80) & \widetilde{\g}^{(5,\overline{2})}=(0,2,76,48)
%\end{array}$$
% $(0,4)$, $(0,2)$ and $(0,1,4)$ are $1$-good for $(1,4)$.
 Now suppose that the claims of proposition hold for $d-1$ and we will prove it for $d$.
 \textbf{Case 1($n$ odd):} Let $n=2d+1$ and consider $\g^{(n,j)}$ for some $1<j \leq (n+1)/2$. We have to show that $\g^{(n,j)}$ is $d$-good for $\g^{(n,1)}$.
 For special case, $j=(n+1)/2$, we have by Lemma \ref{Re}$(1)$ $$\g^{(n,(n+1)/2)}=\sum_{k=1}^{(n-1)/2}(\ga^{(n-1,k)}+\g^{(n-1,\overline{k})}).$$ Since $n-1=2(d-1)+2 $ is even,
 by induction hypothesis, each summand is $d$-good for $\g^{(n-1,1)}$, and hence by Lemma \ref{5.5}, $\g^{(n,(n+1)/2)}$ is $d$-good for $\g^{(n,1)}$.
 For  $1<j < (n+1)/2$, by  Lemma \ref{Re}, we have  $$\g^{(n,j)}=2\sum_{k=1}^{\lfloor n/2\rfloor}\g^{(n-1,\overline{k})}+2\sum_{k=1}^{j-1}\ga^{(n-1,k)}+\sum_{k=j}^{\lfloor n/2\rfloor}\g^{(n-1,{k})}.$$
 By induction hypothesis, first two  summands in $\g^{(n,j)}$ are  $d$-good for $\g^{(n-1,1)}$ and the last summand is $(d-1)$-good for $\g^{(n-1,1)}$ followed from \ref{Obs}.
 Thus, the sum of $d$ and $(d-1)$-good vectors is $d$-good vector, so $\g^{(n,j)}$ is $d$-good for $\g^{(n-1,1)}$ and hence also $d$-good for $\g^{(n,1)}$ by Lemma \ref{5.5}, proving part 1(a).\\

 Now, we want to show that $\g^{(n,\overline{j})}$, for $1\leq j\leq (n+1)/2$ is $d$-good for $\g^{(n,1)}$. By Lemma \ref{Re}, for $1\leq j< (n+1)/2$
 $$\g^{(n,\overline{j})}=2\sum_{k=1}^{\lfloor n/2\rfloor}(0,\g^{(n-1,k)})+2\sum_{k=1}^{j-1}\ga^{(n-1,\overline{k})}+\sum_{k=j}^{\lfloor n/2\rfloor}\g^{(n-1,\overline{k})}.$$
The last two summands by induction are $d$-good for $\g^{(n-1,1)}$,
hence also $d$-good for $\g^{(n,1)}$ by Lemma \ref{5.5}. The the
first sum term can be rewritten as $(0,2\alpha)$, where
$$\alpha=\sum_{k=1}^{\lfloor n/2\rfloor}\g^{(n-1,k)}$$ which is by
induction an  $(d-1)$-FFK vector and by Lemma \ref{Re}(3)
$$\g^{(n,1)}=\sum_{k=1}^{\lfloor
n/2\rfloor}(\g^{(n-1,k)}+2\g^{(n-1,\overline{k})})$$  dominates
$\alpha$. Thus, by Definition \ref{good}(1), $(0,2\alpha)$ is $d$-good
for $\g^{(n,1)}$. Thus, $\g^{(n,\overline{j})}$ is $d$-good for
$\g^{(n,1)}$. For $j=\frac{n+1}{2}$, we have
$$\g^{(n,\overline{(n+1)/2)}}=\sum_{k=1}^{(n-1)/2}[\ga^{(n-1,\overline{k})}+(0,\g^{(n-1,k)})].$$
$\g^{(n,\overline{(n+1)/2})}$ is $d$-good for $\g^{(n,1)}$ as all
terms are $d$-good for $\g^{(n,1)}$ followed by induction hypothesis
and Lemma \ref{5.5}.

For part (1c), we know from Lemma \ref{Re}:
$$\ga^{(n,j)}=\sum_{k=1}^{j-1}\ga^{(n-1,k)}+2\sum_{k=j}^{\lfloor
n/2\rfloor}(0,\g^{(n-1,{k})})+\sum_{k=1}^{\lfloor
n/2\rfloor}\g^{(n-1,\overline{k})},$$ for $1\leq j<(n+1)/2$.
 In the degenerate case $j=1$, this gives $$\ga^{(n,1)}=2\sum_{k=1}^{\lfloor n/2\rfloor}(0,\g^{(n-1,{k})})+\sum_{k=1}^{\lfloor n/2\rfloor}\g^{(n-1,\overline{k})}=\g^{(n,\overline{1})}$$
which is $d$-good  for $\g^{(n,1)}$. If $j>1$, we know from Lemma
\ref{Re}(5) that:$$\ga^{(n,j)}=
\sum_{k=1}^{j-1}\ga^{(n-1,k)}+2\sum_{k=j}^{\lfloor
n/2\rfloor}(0,\g^{(n-1,{k})})+\sum_{k=1}^{\lfloor
n/2\rfloor}\g^{(n-1,\overline{k})}$$ that can be further rewritten
as: $$=\sum_{k=2}^{j-1}\ga^{(n-1,k)}+\sum_{k=1}^{\lfloor
n/2\rfloor}\g^{(n-1,{\overline{k}})}+2\sum_{k=j}^{\lfloor
n/2\rfloor}(0,\g^{(n-1,{k})})+2\sum_{k=1}^{\lfloor
n/2\rfloor}(0,\g^{(n-1,{k})}) +\sum_{k=1}^{\lfloor
n/2\rfloor}\g^{(n-1,\overline{k})}$$
$$\ga^{(n,j)}=\sum_{k=2}^{j-1}\ga^{(n-1,k)}+2\sum_{k=1}^{\lfloor n/2\rfloor}\g^{(n-1,\overline{k})}+(0,2\beta),$$

where $$\beta=\sum_{k=j}^{\lfloor n/2\rfloor}\g^{(n-1,{k})}+\sum_{k=1}^{\lfloor n/2\rfloor}\g^{(n-1,{k})}.$$
The first two sums are $d$-good for $\g^{(n-1,1)}$ by induction and hence by Lemma \ref{5.5},
are $d$-good for $\g^{(n,1)}$. Also by induction $\beta$ is $(d-1)$-FFK vector and $\g^{(n,1)}$ dominates $\beta$.
Thus, by Definition \ref{good}(1), $(0,2\beta)$ is $d$-good for $\g^{(n,1)}$. Hence, $\ga^{(n,j)}$ is $d$-good for $\g^{(n,1)}$.\\

\noindent For part (1d), we have
$$\ga^{(n,\overline{j})}=\sum_{k=1}^{j-1}\ga^{(n-1,\overline{k})}+2\sum_{k=j}^{\lfloor
n/2\rfloor}(0,\g^{(n-1,\overline{k})})+\sum_{k=1}^{\lfloor
n/2\rfloor}(0,\g^{(n-1,{k})}),$$ for $1\leq j<(n+1)/2$. In the
degenerate case $j=1$, this gives
$\ga^{(n,\overline{1})}=(0,\g^{(n,1)})$ which is clearly
$(d+1)$-good for $\g^{(n,1)}$.
 If $j>1$, we can rewrite it by using Lemma \ref{Re} (5):
$$\ga^{(n,\overline{j})}=\sum_{k=2}^{j-1}\ga^{(n-1,k)}+(0,\g^{(n,1)})+(0,\eta),$$ where
$$\eta=2\sum_{k=j}^{\lfloor n/2\rfloor}\g^{(n-1,\overline{k})}+\sum_{k=1}^{\lfloor n/2\rfloor}\g^{(n-1,{k})}.$$
The first  sum is $d$-good for $\g^{(n-1,1)}$ by induction and hence by Lemma \ref{5.5}, is $d$-good for $\g^{(n,1)}$.
Also by induction, $\eta$ is $(d-1)$-FFK vector and $\g^{(n,1)}$ dominates $\eta$. Thus, by Definition \ref{good}, $(0,2\eta)$ is $d$-good for $\g^{(n,1)}$.
But $(0,\g^{(n,1)})$ is $(d+1)$-good for $\g^{(n,1)}$.  Hence, by Observation \ref{3.7}, the sum of $d$-good and $(d+1)$-good is $(d+1)$-good. Hence,  $\ga^{(n,j)}$ is $(d+1)$-good for $\g^{(n,1)}$.\\
\textbf{Case 2($n$ even).} Let $n=2d+2$. For $1 <j\leq /2$, we have
$$\g^{(n,j)}=2\sum_{k=1}^{\lfloor
n/2\rfloor}\g^{(n-1,\overline{k})}+2\sum_{k=1}^{j-1}\ga^{(n-1,k)}+\sum_{k=j}^{\lfloor
n/2\rfloor}\g^{(n-1,{k})}.$$ Sine $n-1=2d+1$ so by Case 1, each term
in  all summations is $d$-good for $\g^{(n-1,1)}$ and hence, by Lemma
\ref{5.5} we get $\g^{(n,j)}$ is $d$-good for $\g^{(n,1)}$.

For part (2b), by Lemma \ref{Re}(3),
$$\g^{(n,\overline{j})}=2\sum_{k=1}^{\lfloor
n/2\rfloor}(0,\g^{(n-1,k)})+2\sum_{k=1}^{j-1}\ga^{(n-1,\overline{k})}+\sum_{k=j}^{\lfloor
n/2\rfloor}\g^{(n-1,\overline{k})}.$$ The first sum is $d$-good for
$\g^{(n,1)}$ followed from Case $1$. The second sum is $(d+1)$-good
for $\g^{(n-1,1)}$ and the last sum is $d$-good for $\g^{(n-1,1)}$
by induction. Therefore,  the sum of  last two summations is
$(d+1)$-good for $\g^{(n-1,1)}$ by Observation \ref{3.7}, hence, is
$(d+1)$-good for $\g^{(n,1)}$  by Lemma \ref{5.5}. Thus, $\g^{(n,\overline{j})}$ is $(d+1)$-good for $\g^{(n,1)}$.
The proofs of (2c) and (2d) trivially follow from (1c) and (1d).
\end{proof}

Proposition \ref{mainprop} and the above discussion show that $\g(\Int(\D))$ is  $d$-good for $\g^{(n+1,1)}$, i.e., $d$-FFK-vector, hence is $f$-vector of a balanced simplicial complex which completes the proof of Theorem \ref{main}.\\

 Since $\g(\Int(\D))$ is nonnegative so here arises a natural
 question.
\begin{ques}
  What is the combinatorial interpretation for the $\g$-vectors $\g^{(n,j)}$?
\end{ques}
We conclude this section with the following remark:
\begin{remark}
 As $\g(\Int(\D))$ is $f$-vector of a simplicial complex so it would be interesting if one could explicitly construct the simplicial complex $\Gamma(\Int(\D))$ such that $f(\Gamma(\Int(\D)))=\g(\Int(\D))$.
\end{remark}
\section{Local $\g$-vector}\label{last}
 Since $B^{++}_n(x)$ and $B^{-+}_n(x)$ are symmetric polynomials and $$B^+_n(x)=B^{++}_n(x)+B^{-+}_n(x),$$
 so $B^{+}_n(x)$ decomposes as sum of two nonnegative, symmetric polynomials. In this section,
 we show that the polynomials $B^{++}_n(x)$ and $B^{-+}_n(x)$ are both $\g$-nonnegative
 by giving a combinatorial description. Moreover, a geometric interpretation of the  polynomial $B^{++}_n(x)$  as
 the $h$-polynomial  $h(\partial(\Int(2^{[n]})),x)$ of  the boundary of interval subdivision of a simplex is given.
\subsection{A Geometrical Interpretation:} In this subsection, we show that $B^{++}_n(x)$ is the $h$-polynomial of the boundary of interval subdivision of a simplex $2^{[n]}$.  Let $\Gamma$ be the interval subdivision of a simplex $2^{[n]}$.

\begin{prop}\label{boundary}
  We have $h(\partial(\Gamma),x)=B^{++}_n(x).$
\end{prop}

\begin{proof}
Viewing $\partial(\Gamma)$ as interval subdivision of $\partial(2^{[n]})$ and by Theorem \ref{3.1}, we have $$h_r(\partial(\Gamma))=\sum_{s=0}^{n-1}B^+(n,s+1,r)h_s(\partial(2^{[n]})).$$
  Using the fact that   $h(\partial(2^{[n]}) )=\left(
                                                      \begin{array}{c}
                                                        1 \\
                                                        1\\
                                                        \vdots \\
                                                        1\\
                                                      \end{array}
                                                    \right),$
   we get $$ h_r(\partial(\Gamma))=\sum_{s=0}^{n-1}B^+(n,s+1,r)=B^{++}(n,r).$$
\end{proof}
\subsection{The type $B$ Expansion:} Now we are turning toward analyzing the behavior of type $B$ expansion. Let us first recall, the
definition of slides. Let $\si=\si_1\dots \si_n\in B_n$  and
consider $\si_0\si_1\dots\si_n\si_{n+1}$, where $\si_0=0$ and
$\si_{n+1}=\bar{\infty}$. Put asterisks at each end and also between
$\si_i$ and $\si_{i+1}$ whenever $\si_i<\si_{i+1}$.  {\em A  slide
is any segment between asterisks of length at least $2$.} In other
words, a  slide of $\si$ is any decreasing run of
$\si_0\si_1\dots\si_n\si_{n+1}$ of length at least $2$. For example,
for the permutation $3\bar{5}71\bar{6}8\bar{9}\bar{4}2\in B_9$,
$*0*3\bar{5}*71\bar{6}*8\bar{9}*\bar{4}*2\bar{\infty}*$ there are  four slides, namely, $3\bar{5}, 71\bar{6}, 8\bar{9}$ and $2\bar{\infty}$.\\

Let $W_B(n, k, s) := \{\si\in B_n: \si \hbox{ has }k \hbox{ descents
and } s+1 \hbox{ slides}\}$, $W_B(n, s) := W_B(n, s, s)$. It follows
that every slide of an element  of
$W_B(n, s)$ must be of length exactly $2$. Also notice that $k\geq s$ as each slide gives at least $1$ descent except the last slide(the last slide may or may not have a descent). \\
Lets denote $b^{++}(n,k,s) := |W_B(n, k, s)\cap B^{++}_n|$, $b^{-+}(n,k,s) := |W_B(n, k, s)\cap B^{-+}_n|$, $b^{++}(n,s) := b^{++}(n,s,s)$ and $b^{-+}(n,s) := b^{-+}(n,s,s)$. It can be observed that an element $\si$ in $B^{++}_n$ has at least $1$ slide and an element $\si$ in $B^{-+}_n$ has at least $2$ slides.
\begin{prop}\label{slide}
  We have  $$b^{++}(n,k,s)=\binom{n-1-2s}{k-s} b^{++}(n,s)\hbox{ and }b^{-+}(n,k,s)=\binom{n-1-2s}{k-s} b^{-+}(n,s)$$ which give the following relations
    $$B^{++}(n,k)=\sum_{s=0}^{k}\binom{n-1-2s}{k-s} b^{++}(n,s) \hbox{ and }B^{-+}(n,k)=\sum_{s=1}^{k}\binom{n-1-2s}{k-s} b^{-+}(n,s).$$

\end{prop}

\begin{proof}

  Lets prove the relations for $B^{++}_n$. Let $\si=\si_1\ldots \si_n \in B^{++}_n\cap W_B(n,s,s)$. Counting $\si_{n+1}=\bar{\infty}$, there are $n + 1$ symbols and $n + 1- 2(s+1)=n-1-2s$ that are not included in the slides. Choose $k-s$ of these $n - 1-2s$
elements, move each chosen element $\si_k$ to the left if $\si_k<0$( to right  if $\si_k>0$, respectively ) into the nearest slide $*\si_j\si_{j+l}*$
with $\si_j >\si_k>\si_{j+l}.$  After moving chosen elements, the resulting permutation is still in $B^{++}_n$. Thus, the first relation holds. The second assertion follows upon summing $b^{++}(n,k,s)$ over $0\leq s\leq k.$ For $B^{-+}_n$, the proof is  similar.

Figure $1$ illustrates the one-one correspondence. For example, let  $56\bar{8}\bar{1}2\bar{9}\bar{7}34\in B^{++}_9$,  $*0*5*6\bar{8}*\bar{1}*2\bar{9}*\bar{7}*3*4\bar{\infty}*$ has 3 slides and 2 descents.
\begin{figure}[htbp]

{\includegraphics[width=0.7\linewidth]{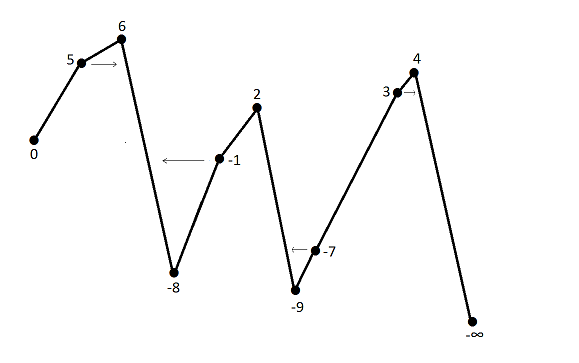}}

\caption{}
\end{figure}
\end{proof}
\begin{teo}\label{5.2}
  For $n\geq 1$, $$B^{++}_n(x)=\sum_{s=0}^{\lfloor\frac{n-1}{2}\rfloor}b^{++}(n,s)x^s(1+x)^{n-1-2s}$$ and $$B^{-+}_n(x)=\sum_{s=1}^{\lfloor\frac{n-1}{2}\rfloor}b^{-+}(n,s)x^s(1+x)^{n-1-2s}.$$
\end{teo}
\begin{proof}
  The result follows by Proportion \ref{slide}, the relation \eqref{sym} and the fact that $\binom{n-1-2s}{k-s}=\binom{n-1-2s}{n-1-k-s}$.
\end{proof}

\begin{remark}

 By  Theorem \ref{3.1}, we have
 \begin{equation}\label{BB+}
 B^+_n(x)=h(\Int(2^{[n]}), x),
 \end{equation}
   so $B^{-+}_n(x)$ is equal to the difference of $h(\Int(2^{[n]}),x)$ and $h(\partial(\Int(2^{[n]})),x)$. Thus, by Theorem \ref{5.2}, $B^{-+}_n(x)$ is is nonnegative, $\g$-nonnegative, and hence unimodal. Hence, by Theorem  \ref{Juhnke},  the local $h$-vector of $\Gamma$  can be written  as
   $$\ell_V(\Gamma, x)=\sum_{k=0}^{n}{n \choose k}B^{-+}_n(x)\mathfrak{d}_{n-k}^A(x).$$ Since the sum and product of $\g$-nonnagative polynomials is $\g$-nonnegative, so the local $\g$-vector of $\Gamma$ is nonnegative, which proves Theorem \ref{local gamma}.

\end{remark}

\bibliographystyle{amsalpha}
\bibliography{References}
\end{document}